\documentclass[a4paper,11pt,oneside]{article}
\usepackage{graphicx}
\usepackage{amscd}
\usepackage{amsmath}
\usepackage{caption}
\usepackage{amsfonts}
\usepackage{amssymb}
\usepackage{amsthm}
\usepackage{mathrsfs}
\usepackage{multicol}
\usepackage{color}
\usepackage[english]{babel}
\usepackage[T1]{fontenc}
\usepackage{textcomp}
\usepackage[utf8]{inputenc}
\usepackage{enumitem}
\usepackage{indentfirst}

\newcommand{\N}{\mathbb N}

\newcommand{\RR}{{{\rm I} \kern -.15em {\rm R} }}

\begin{document}
	\theoremstyle{plain} \newtheorem{thm}{Theorem}[section] \newtheorem{cor}[thm]{Corollary} \newtheorem{lem}[thm]{Lemma} \newtheorem{prop}[thm]{Proposition} \theoremstyle{definition} \newtheorem{defn}{Definition}[section] 
	
	\newtheorem{oss}[thm]{Remark}
	\newtheorem{ex}{Example}[section]
	\newtheorem{lemma}{Lemma}[section]
	\title{Convergence to consensus results for Hegselmann-Krause type models with attractive-lacking interaction}
	\author{{\sc Elisa Continelli \& Cristina Pignotti}
		\\
		Dipartimento di Ingegneria e Scienze dell'Informazione e Matematica\\
		Universit\`{a} degli Studi di L'Aquila\\
		Via Vetoio, Loc. Coppito, 67100 L'Aquila Italy}

	\maketitle

	\begin{abstract}
		In this paper, we analyze a Hegselmann-Krause opinion formation model with attractive-lacking interaction. More precisely, we investigate the situation in which the individuals involved in an opinion formation process interact among themselves but can eventually suspend the exchange of information among each other at some times. Under quite general assumptions, we prove the exponential convergence to consensus for the Hegselmann-Krause model in presence of possible lack of interaction. We then extend the analysis to an analogous model in presence of time delays.
	\end{abstract}

	\vspace{5 mm}
	
	\section{Introduction}
	The study of multiagent systems is a current topic in the literature and finds a large application in many scientific fields, among them biology \cite{Cama, CS1}, economics \cite{Marsan}, robotics \cite{Bullo, Jad}, control theory \cite{Aydogdu, Borzi, WCB, PRT, Piccoli}, social sciences \cite{Bellomo, Campi, CF}. In 2002, Hegselmann and Krause formulated a model (see \cite{HK}), the so-called  Hegselmann-Krause model, in which the approach to consensus for a group of individuals, among whom some opinion formation process is taking place, is investigated. Later on, the second-order version of the Hegselmann-Krause model was introduced by Cucker and Smale in \cite{CS1} for the description of flocking phenomena (for instance, flocking of birds or schooling of fishes).

A possible scenario that could occur in the investigation of such models is the one in which the particles' systems suspend the interactions they have with the other agents. As a consequence, there is a lack of connection between the elements of the system that obstacles the convergence to consensus or the exhibition of asymptotic flocking. In \cite{Bonnet}, the convergence to consensus and the asymptotic flocking for a class of Cucker-Smale systems in the case of communication failures among the system's agents have been proved.

Inspired by this, we provide a new method, different with respect to the one employed by the authors of \cite{Bonnet}, to prove the convergence to consensus for a class of Cucker-Smale systems of the first-order, i.e. for a class of Hegselmann-Krause systems. Our approach is effective for Hegselmann-Krause type models with possible lack of interaction among the system's agents also  in presence of time delays. Moreover, we are able to deal with a very general class of influence functions, namely functions that are only positive, continuous and bounded, without simmetry or monotonicity requirements.  

In this paper, we analyze the asymptotic behaviour of the solutions of a Hegselmann-Krause model involving a nonnegative weight function. More precisely, we deal with the following Hegselmann-Krause type opinion formation model $$\frac{d}{dt}x_{i}(t)=\sum_{j=1}^{N}\alpha(t) a_{ij}(t)(x_{j}(t)-x_{i}(t)),\quad t>0,	\,\, \forall i=1,\dots,N,$$
where a weight function $\alpha$, which is assumed to take values in the interval $[0,1]$ and to be continuous, is included in the classical formulation. Under quite general assumptions, we are able to prove that, if the integrals of the weight function $\alpha$ over some intervals are bounded from below by a certain positive constant, then all the solutions of the considered Hegselmann-Krause model converge exponentially to consensus.

As already pointed out, a particular case of the analysis carried out in this paper is the one in which the weight function $\alpha$ is identically zero in certain suitable intervals. In this situation, we have a sort of "on-off" influence, so that the particles of the system do not interact with each other for some time, that, intuitively, cannot be too large. However, the system's agent exchange information among themselves enough to get consensus.  

In the investigation of such models, it is reasonable the introduction of some time delays. Indeed, in the applications, it often occurs that the information the agent of the system exchange among themselves does not come "promptly". 

The analysis of the Hegselmann-Krause model and the Cucker-Smale model in presence of a time delay (delay that can be constant or, more realistically, variable with time), has been carried out by many authors, \cite{LW, CH, CL, CP, PT, HM, DH, Lu, CPP, P, H, H3}. Most of them require an upper bound on the time delay size in order to achieve te convergence to consensus (in the case of the opinion formation Hegselmann-Krause model) or the exhibition of asymptotic flocking (in the case of the Cucker-Smale flocking model). However, very recently, Rodriguez Cartabia proved in \cite{Cartabia} the asymptotic flocking for the Cucker-Smale model with constant time delay without requiring any smallness assumption on the time delay size (see also \cite{H3} for a consensus result for the Hegselmann-Krause model). Inspired by the work of Rodriguez Cartabia, we were able to obtain the convergence to consensus for the Hegselmann-Krause model with time-variable time delay (see \cite{ContPign}) and the exhibition of asymptotic flocking for the Cucker-Smale model with time-variable time delay (see \cite{Cont}) without asking the delay to be small. Also, the two aforementioned papers improve several previous results by removing the monotonicity of the influence function, which is usually assumed in the study of such multiagent systems. 

Motivated by this, we extend our analysis to a Hegselmann-Krause model with time delay and attractive-lacking interaction, i.e. of the form 
$$\frac{d}{dt}x_{i}(t)=\underset{j:j\neq i}{\sum}\alpha(t) a_{ij}(t)(x_{j}(t-\tau)-x_{i}(t)),\quad t>0,	\,\, \forall i=1,\dots,N,$$
where the weight function $\alpha$ is analogous to the one considered in the case of the Hegselmann-Krause model with attractive-lacking interaction in absence of time delays. Generalizing the proof of the exponential convergence to consensus employed for the Hegselmann-Krause model with a nonnegative weight function and adopting a step-by-step procedure in the spirit of \cite{Cont}, \cite{ContPign}, \cite{Cartabia}, we are able to prove the exponential convergence to consensus for the above model without requiring smallness assumptions on the time delay size $\tau$.

The present paper is so organized. In Section 2, we prove exponential convergence to consensus for the considered Hegselmann-Krause model in the particular case in which no time delay in the dynamics of our multiagent system is included. The reason why we first give a convergence to consensus result in the undelayed case is that the proof of such a result is much more straightforward than the proof one has to employ in the retarded case. 
	
	In Section 3, we establish the exponential convergence to consensus for the solutions of the Hegselmann-Krause model in the delayed case. In this situation, in the spirit of \cite{Cont}, \cite{ContPign}, and \cite{Cartabia}, we prove the decay in time of the diameters of the opinions. This analysis includes of course the one carried out in Section 2, when $\tau=0$, but, however, the techniques one has to employ to prove the exponential decay estimate for the opinion diameters in presence of a time delay do not allow to obtain a so accurate result as the one in Section 2. More precisely, the estimate provided by Theorem \ref{consgen}, in the case in which the time delay is zero is "worst" than the one established in \ref{consnd}. This enhances not redundant the study conducted in the previous section.

	Next, in Section 4 we deal with the continuum model, i.e. the model obtained as the mean-field limit of the particle model when the number of individuals in the opinion formation process tends to infinity. The independence of the number of agents in the estimates found in the previous results is crucial to prove the convergence to consensus also for the continuum model.
	\section{The Hegselmann-Krause model with attractive-lacking interaction}
	Consider a finite set of $N\in\N$ agents, with $N\geq 2 $. Let $x_{i}(t)\in \RR^d$ be the opinion of the $i$-th agent at time $t$. We shall denote with $\lvert\cdot \rvert$ and $\langle\cdot,\cdot\rangle$ the usual norm and scalar product on $\RR^{d}$, respectively. The interactions between the elements of the system are described by the following Hegselmann-Krause type model:
	\begin{equation}\label{onoffnd}
		\frac{d}{dt}x_{i}(t)=\sum_{j=1}^N\alpha(t) a_{ij}(t)(x_{j}(t)-x_{i}(t)),\quad t>0,	\,\, \forall i=1,\dots,N,
	\end{equation}
with the initial conditions \begin{equation}\label{incondnd}
	x_{i}(0)=x_{i}^{0}\in \RR^{d}, \quad\forall i=1,\dots,N.
\end{equation}
	Here, the communication rates $a_{ij}$ are of the form
	\begin{equation}\label{weightnd}
		a_{ij}(t):=\frac{1}{N-1}\psi(x_{i}(t), x_{j}(t)), \quad\forall t>0,\, \forall i,j=1,\dots,N,
	\end{equation}
	where $\psi:\RR^d\times\RR^d\rightarrow \RR$ is a positive, bounded, and continuous function called influence function. In the sequel, we will denote with
	\begin{equation}\label{Knd}
		K:=\lVert \psi\rVert_{\infty},
	\end{equation}
\begin{equation}\label{maxincond}
	M^{0}=\max_{i=1,\dots,N}\lvert x_{i}(0)\rvert,
\end{equation}
\begin{equation}\label{psi0nd}
	\psi_{0}:=\min_{\vert y\vert, \vert z\vert \le M^{0}}\psi(y,z).
\end{equation}
In addition, the weight function $\alpha:[0,+\infty)\rightarrow [0,1]$ is continuous and satisfies the following condition
\begin{itemize}
	\item[(WF)] there exist $T>0$ and a sequence $\{t_{n}\}_{n}$ of nonnegative numbers such that $t_{0}=0$, $t_{n}\to\infty$ as $n\to\infty$, 
	\begin{equation}\label{tnnd}
		t_{n}-t_{n-1}\leq T,\quad \forall n\in \mathbb{N},
	\end{equation}
	and 
	\begin{equation}\label{alpha1}
		\int_{t_{n-1}}^{t_{n}}\alpha(t)\,dt\geq \bar{\alpha},\quad \forall n\in\mathbb{N},
	\end{equation}
	for some positive constant $\bar{\alpha}$.
\end{itemize}
	For each $t\geq0$, we define the diameter $d(\cdot)$ as
	$$d(t):=\max_{i,j=1,\dots,N}\lvert x_{i}(t)-x_{j}(t)\rvert.$$
	
	\begin{defn}\label{consensus}
		We say that a solution $\{x_{i}\}_{i=1,\dots,N}$ to system \eqref{onoffnd} converges to \textit{consensus} if $d(t)\to0$, as $t\to \infty$.
	\end{defn}
\begin{thm}\label{consnd}
	Let $\psi:\RR^d\times \RR^d\rightarrow\RR$ be a positive, bounded, continuous function. Assume that the weight function $\alpha:[0,+\infty)\rightarrow[0,1]$ satisfies $(WF)$. Then, every solution $\{x_{i}\}_{i=1,\dots,N}$ to \eqref{onoffnd} with the initial conditions \eqref{incondnd} satisfies the exponential decay estimate 
	\begin{equation}\label{estgennd}
		d(t)\leq d(0)e^{-\gamma(t-T)},\quad \forall  t\geq 0,
	\end{equation}
	where $T$ is the positive constant in \eqref{tnnd}, for a suitable positive constant $\gamma$ independent of $N.$
\end{thm}
\subsection{Proof of Theorem \ref{consnd}}
Let $\{x_{i}\}_{i=1,\dots,N}$ be solution to \eqref{onoffnd} under the initial conditions \eqref{incondnd}. We assume that the hypotheses of Theorem \ref{consnd} are satisfied. For the proof of Theorem \ref{consnd} we will need the following results. In the sequel, we will denote  $\mathbb{N}_0:=\mathbb{N}\cup \{0\}.$
\begin{lem}\label{L1nd}
	For each $v\in \RR^{d}$ and $S\geq 0,$  we have that 
	\begin{equation}\label{scalprnd}
		\min_{j=1,\dots,N}\langle x_{j}(S),v\rangle\leq \langle x_{i}(t),v\rangle\leq \max_{j=1,\dots,N}\langle x_{j}(S),v\rangle,
	\end{equation}for all  $t\geq S$ and $i=1,\dots,N$. 
\end{lem}
\begin{proof}
	Let $S\geq 0$. Given a vector $v\in \RR^{d}$, we set $$M_S=\max_{j=1,\dots,N}\langle x_{j}(S),v\rangle.$$
	For all $\epsilon >0$, we define
	$$K^{\epsilon}:=\left\{t>S :\max_{i=1,\dots,N}\langle x_{i}(s),v\rangle < M_T+\epsilon,\,\forall s\in [S,t)\right\}.$$
	By continuity, $K^{\epsilon}\neq\emptyset$ and, denoted with $$S^{\epsilon}:=\sup K^{\epsilon},$$
	it holds $S^{\epsilon}>T$. \\We claim that $S^{\epsilon}=+\infty$. Indeed, suppose by contradiction that $S^{\epsilon}<+\infty$. By definition of $S^{\epsilon}$, it turns out that \begin{equation}\label{maxnd}
		\max_{i=1,\dots,N}\langle x_{i}(t),v\rangle<M_S+\epsilon,\quad \forall t\in (S,S^{\epsilon}),
	\end{equation}
	\begin{equation}\label{tepsnd}
		\lim_{t\to S^{\epsilon-}}\max_{i=1,\dots,N}\langle x_{i}(t),v\rangle=M_S+\epsilon.
	\end{equation}
	For all $i=1,\dots,N$ and $t\in (S,S^{\epsilon})$, using \eqref{Knd}, \eqref{maxnd} and the fact that $\alpha(t)\in [0,1]$, we have that
	$$\begin{array}{l}
		\vspace{0.3cm}\displaystyle{\frac{d}{dt}\langle x_{i}(t),v\rangle=\frac{1}{N-1}\sum_{j:j\neq i}\alpha(t)\psi(x_{i}(t), x_{j}(t))\langle x_{j}(t)-x_{i}(t),v\rangle}\\
		\vspace{0.3cm}\displaystyle{\hspace{2cm}\leq \frac{1}{N-1}\sum_{j:j\neq i}\alpha(t)\psi(x_{i}(t), x_{j}(t))(M_T+\epsilon-\langle x_{i}(t),v\rangle)}\\
		\displaystyle{\hspace{2cm}\leq K(M_S+\epsilon-\langle x_{i}(t),v\rangle).}
	\end{array}$$
	Thus, applying Gronwall's inequality we get
	$$\begin{array}{l}
		\vspace{0.2cm}\displaystyle{
			\langle x_{i}(t),v\rangle\leq e^{-K(t-S)}\langle x_{i}(S),v\rangle+K(M_S+\epsilon)\int_{S}^{t}e^{-K(t-s)}ds}\\
		\vspace{0.3cm}\displaystyle{\hspace{1.7 cm}
			=e^{-K(t-S)}\langle x_{i}(S),v\rangle+(M_S+\epsilon)e^{-Kt}(e^{Kt}-e^{KS})}\\
		\vspace{0.3cm}\displaystyle{\hspace{1.7 cm}
			=e^{-K(t-S)}\langle x_{i}(S),v\rangle+(M_S+\epsilon)(1-e^{-K(t-S)})}\\
		\vspace{0.3cm}\displaystyle{\hspace{1.7 cm}
			\leq e^{-K(t-S)}M_S+M_S+\epsilon -M_Se^{-K(t-S)}-\epsilon e^{-K(t-S)}}\\
		\vspace{0.3cm}\displaystyle{\hspace{1.7 cm}
			=M_S+\epsilon-\epsilon e^{-K(t-S)}}\\
		\displaystyle{\hspace{1.7 cm}
			\leq M_S+\epsilon-\epsilon e^{-K(S^{\epsilon}-S)},}
	\end{array}
	$$for all $t\in (S, S^{\epsilon})$.	Therefore, $\forall i=1,\dots, N,$
	$$\langle x_{i}(t),v\rangle\leq M_S+\epsilon-\epsilon e^{-K(S^{\epsilon}-S)}, \quad \forall t\in (S,S^{\epsilon}),$$
	from which
	\begin{equation}\label{limnd}
		\max_{i=1,\dots,N} \langle x_{i}(t),v\rangle\leq M_S+\epsilon-\epsilon e^{-K(S^{\epsilon}-S)}, \quad \forall t\in (S,S^{\epsilon}).
	\end{equation}
	Letting $t\to S^{\epsilon-}$ in \eqref{limnd}, we have that 
$$\lim_{t\to S^{\epsilon-}}\max_{i=1,\dots,N}\langle x_{i}(t),v\rangle\leq M_S+\epsilon-\epsilon e^{-K(S^{\epsilon}-S)}<M_S+\epsilon,$$
	in contradiction  with \eqref{tepsnd}. Thus, $S^{\epsilon}=+\infty$ and $$\max_{i=1,\dots,N}\langle x_{i}(t),v\rangle<M_S+\epsilon, \quad \forall t>S.$$
	From the arbitrariness of $\epsilon$ we can conclude that $$\max_{i=1,\dots,N}\langle x_{i}(t),v\rangle\leq M_S, \quad \forall t>S,$$
	from which $$\langle x_{i}(t),v\rangle\leq M_S, \quad \forall t>S, \,\forall i=1,\dots,N.$$
	So, the second inequality in \eqref{scalprnd} is proved. Now, we prove the first inequality in \eqref{scalprnd}. Given $v\in \RR^{d}$, we define $$m_S=\min_{j=1,\dots,N}\langle x_{j}(S),v\rangle.$$
	Then, for all $i=1,\dots,N$ and $t>S$, by applying the second inequality in \eqref{scalprnd} to the vector $-v\in\RR^{d}$ we get $$-\langle x_{j}(s),v\rangle=\langle x_{i}(t),-v\rangle\leq \max_{j=1,\dots,N}\langle x_{j}(S),-v\rangle$$$$=-\min_{j=1,\dots,N}\langle x_{j}(S),v\rangle=-m_S,$$
	from which $$\langle x_{j}(s),v\rangle\geq m_S.$$
	Thus, also the first inequality in \eqref{scalprnd} holds.
\end{proof}
\begin{lem}
	For each $n\in \mathbb{N}_{0}$ and $i,j=1,\dots,N$, we get \begin{equation}\label{distgennd}
		\lvert x_{i}(s)-x_{j}(t)\rvert\leq d(t_{n}), \quad\forall s,t\geq t_{n}.
	\end{equation} 
\end{lem}
\begin{proof}
	Fix $n\in\mathbb{N}_{0}$ and $i,j=1,\dots,N$. Given $s,t\geq t_{n}$, if $\lvert x_{i}(s)-x_{j}(t)\rvert=0$, then of course $d(t_{n})\geq \lvert x_{i}(s)-x_{j}(t)\rvert$. Thus, we can assume $\lvert x_{i}(s)-x_{j}(t)\rvert>0$. We define the unit vector $$v=\frac{x_{i}(s)-x_{j}(t)}{\lvert x_{i}(s)-x_{j}(t)\rvert}.$$
Then, applying \eqref{scalprnd} with $S=t_{n}$ and the Cauchy-Schwarz inequality, it comes that
	$$\begin{array}{l}
		\vspace{0.3cm}\displaystyle{\lvert x_{i}(s)-x_{j}(t)\rvert=\langle x_{i}(s)-x_{j}(t),v\rangle=\langle x_{i}(s),v\rangle-\langle x_{j}(t),v\rangle}\\
		\vspace{0.3cm}\displaystyle{\leq \max_{l=1,\dots,N}\langle x_{l}(t_{n}),v\rangle-\min_{l=1,\dots,N}\langle x_{l}(t_{n}),v\rangle}\\
		\vspace{0.3cm}\displaystyle{\leq \max_{l,k=1,\dots,N}\langle x_{l}(t_{n})-x_{k}(t_{n}),v\rangle}	\\
		\displaystyle{\leq \max_{l,k=1,\dots,N}\lvert x_{l}(t_{n})-x_{k}(t_{n})\rvert\lvert v\rvert=d(t_{n}).}
	\end{array}$$
\end{proof}
\begin{oss}
	Let us note that \eqref{distgennd} together with the fact that $t_{0}=0$ implies that
	\begin{equation}
		\label{dxnd}
		\lvert x_{i}(s)-x_{j}(t)\rvert\leq d(0), \quad\forall s,t\geq 0.
	\end{equation}
	In addition, we have that
	\begin{equation}\label{decgennd}
		d(t_{n+1})\leq d(t_{n}),\quad \forall n\in \mathbb{N}_{0},
	\end{equation}
that is the sequence of diameters $\{d(t_{n})\}_{n\in \mathbb{N}_{0}}$ is nonincreasing. 
\end{oss} 
Now, we prove that there is a bound on $\vert x_i(t)\vert,$ uniform with respect to $t$ and $i=1,\dots,N.$ 
\begin{lem}\label{L3nd}
	For every $i=1,\dots,N,$ we have that \begin{equation}\label{boundsolnd}
		\lvert x_{i}(t)\rvert\leq M^{0},\quad \forall t\geq0,
	\end{equation}
where $M^{0}$ is the positive constant in \eqref{maxincond}.	
\end{lem}
\begin{proof}
	Given $i=1,\dots,N$ and $t\geq0$, if $\lvert x_{i}(t)\rvert =0$, then trivially $M^{0}\geq \lvert x_{i}(t)\rvert $. On the other hand, if $\lvert x_{i}(t)\rvert >0$, setting $$v=\frac{x_{i}(t)}{\lvert x_{i}(t)\rvert},$$
	it turns out that $v$ is a unit vector and that
	$$\lvert x_{i}(t)\rvert=\langle x_{i}(t),v\rangle. $$
	Now, using \eqref{scalprnd} with $S=0$ and using the Cauchy-Schwarz inequality, we finally get $$\lvert x_{i}(t)\rvert\leq \max_{j=1,\dots,N}\langle x_{j}(0),v\rangle\leq \max_{j=1,\dots,N}\lvert x_{j}(0)\rvert\lvert v\rvert$$$$=\max_{j=1,\dots,N}\lvert x_{j}(0)\rvert=M^{0}.$$
\end{proof}
\begin{lem}
	For all $i,j=1,\dots,N$,  unit vector $v\in \RR^{d}$ and $n\in\mathbb{N}_{0}$, we have
	\begin{equation}\label{4gennd}
		\langle x_{i}(t)-x_{j}(t),v\rangle\leq e^{-K(t-\bar{t})}\langle x_{i}(\bar{t})-x_{j}(\bar{t}),v\rangle+(1-e^{-K(t-\bar{t})})d(t_{n}),
	\end{equation}
	for all $t\geq \bar{t}\geq t_{n}$. 
\end{lem}
\begin{proof}
	Fix $n\in\mathbb{N}_{0}$ and $v\in\RR^{d}$ such that $\lvert v\rvert=1$. We set $$M_n=\max_{i=1,\dots,N}\langle x_{i}(t_{n}),v\rangle,$$$$m_n=\min_{i=1,\dots,N}\langle x_{i}(t_{n}),v\rangle.$$
	Then,  $M_n-m_n\leq d(t_{n})$. Also, for all $i=1,\dots,N$ and $t\geq \bar{t}\geq t_{n}$, from \eqref{Knd}, \eqref{scalprnd} we have that 
	$$
	\begin{array}{l}
		\vspace{0.3cm}\displaystyle{
			\frac{d}{dt}\langle x_{i}(t),v\rangle=\sum_{j:j\neq i}\alpha(t)a_{ij}(t)\langle x_{j}(t)-x_{i}(t),v\rangle}\\
		\vspace{0.3cm}\displaystyle{\hspace{2 cm}\leq \frac{1}{N-1}\sum_{j:j\neq i}\alpha(t)\psi(x_{i}(t), x_{j}(t))(M_n-\langle x_{i}(t),v\rangle)}\\
			\displaystyle{\hspace{2 cm}\leq K(M_n-\langle x_{i}(t),v\rangle)}.
	\end{array}
	$$
	From Gronwall's inequality, it follows that
	\begin{equation}\label{2gennd}
		\langle x_{i}(t),v\rangle\leq e^{-K(t-\bar{t})}\langle x_{i}(\bar{t}),v\rangle+(1-e^{-K(t-\bar{t})})M_n.
	\end{equation}
	Analogously, for all $i=1,\dots,N$ and $t\geq \bar{t}\geq t_{n}$, it holds
	\begin{equation}\label{3gennd}
		\langle x_{i}(t),v\rangle\geq e^{-K(t-\bar{t})}\langle x_{i}(\bar{t}),v\rangle+(1-e^{-K(t-\bar{t})})m_n.
	\end{equation}
	Therefore, for all $i,j=1,\dots,N$ and $t\geq \bar{t}\geq t_{n}$, using \eqref{2gennd} and \eqref{3gennd} we get
	$$
	\begin{array}{l}
		\vspace{0.3cm}\displaystyle{\langle x_{i}(t)-x_{j}(t),v\rangle\leq e^{-K(t-\bar{t})}\langle x_{i}(\bar{t})-x_{j}(\bar{t}),v\rangle+(1-e^{-K(t-\bar{t})})d(t_{n}),}
	\end{array}
	$$
	which concludes our proof.
\end{proof}
\begin{lem}\label{L5gennd}
	There exists  a constant $C\in (0,1),$ independent of $N\in\N,$ such that
	\begin{equation}\label{n-2gennd}
		d(t_{n})\leq C d(t_{n-1}),
	\end{equation}for all $n\geq 1$.
\end{lem}
\begin{proof} 
	Given $n\geq 1$, being inequality \eqref{n-2gennd} trivially satisfied if $d(t_{n})=0$, we can assume that $d(t_{n})>0$. Let $i,j=1,\dots,N$ be such that $d(t_{n})=\lvert x_{i}(t_{n})-x_{j}(t_{n})\rvert$. We set $$v=\frac{x_{i}(t_{n})-x_{j}(t_{n})}{\lvert x_{i}(t_{n})-x_{j}(t_{n})\rvert}.$$
	Then, $v$ is a unit vector for which we can write $$d(t_{n})=\langle x_{i}(t_{n})-x_{j}(t_{n}),v\rangle.$$
	We define $$M_{n-1}=\max_{l=1,\dots,N}\langle x_{l}(t_{n-1}),v\rangle,$$
	$$m_{n-1}=\min_{l=1,\dots,N}\langle x_{l}(t_{n-1}),v\rangle.$$
	Then $M_{n-1}-m_{n-1}\leq d(t_{n-1})$. 
	\\Now, we distinguish two different situations.
	\par\textit{Case I.} Assume that there exists $\bar{t}\in [t_{n-1},t_{n}]$ such that 
	$$\langle x_{i}(\bar{t})-x_{j}(\bar{t}),v\rangle<0.$$
	Then, we can apply \eqref{4gennd} with $t_{n}\geq\bar{t}\geq t_{n-1}$, so that \begin{equation}\label{t0gennd}
		\begin{split}
			d(t_{n})&\leq e^{-K(t_{n}-\bar{t})}\langle x_{i}(\bar{t})-x_{j}(\bar{t}),v\rangle+(1-e^{-K(t_{n}-\bar{t})})d(t_{n-1})\\&\leq (1-e^{-K(t_{n}-\bar{t})})d(t_{n-1})\\&\leq (1-e^{-KT})d(t_{n-1}).
		\end{split}
	\end{equation}
	\par\textit{Case II.} Assume it rather holds \begin{equation}\label{posgennd}
		\langle x_{i}(t)-x_{j}(t),v\rangle\geq 0,\quad \forall t\in [t_{n-1},t_{n}].
	\end{equation}
	Then, for every  $t\in [t_{n-1},t_{n}]$, we have that 
	$$
	\begin{array}{l}
		\displaystyle{
			\frac{d}{dt}\langle x_{i}(t)-x_{j}(t),v\rangle=\frac{1}{N-1}\sum_{l:l\neq i}\alpha(t)\psi(x_{i}(t),x_{l}(t))\langle x_{l}(t)-x_{i}(t),v\rangle}\\
		\displaystyle{\hspace{3.6cm}-\frac{1}{N-1}\sum_{l:l\neq j}\alpha(t)\psi(x_{i}(t),x_{l}(t))\langle x_{l}(t)-x_{j}(t),v\rangle}\\
		\displaystyle{\hspace{2.7cm}=\frac{1}{N-1}\sum_{l:l\neq i}\alpha(t)\psi(x_{i}(t),x_{l}(t))(\langle x_{l}(t),v\rangle-M_{n-1}+M_{n-1}-\langle x_{i}(t),v\rangle)}\\
		\displaystyle{\hspace{3cm}+\frac{1}{N-1}\sum_{l:l\neq j}\alpha(t)\psi(x_{i}(t),x_{l}(t))(\langle x_{j}(t),v\rangle-m_{n-1}+m_{n-1}-\langle x_{l}(t),v\rangle)}\\
		\displaystyle{\hspace{5.5 cm}:=S_1+S_2.}
	\end{array}
	$$
	Now, using \eqref{psi0nd} and \eqref{boundsolnd}, we get
	$$
	\begin{array}{l}
		\displaystyle{
			S_1=\frac{1}{N-1}\sum_{l:l\neq i}\alpha(t)\psi(x_{i}(t),x_{l}(t))(\langle x_{l}(t),v\rangle -M_{n-1})}\\
		\displaystyle{\hspace{0.7 cm}+\frac{1}{N-1}\sum_{l:l\neq i}\alpha(t)\psi(x_{i}(t),x_{l}(t))(M_{n-1}-\langle x_{i}(t),v\rangle)}\\
		\displaystyle{\hspace{0.4 cm}\leq \frac{1}{N-1}\psi_{0}\alpha(t)\sum_{l:l\neq i}(\langle x_{l}(t),v\rangle-M_{n-1})+K(M_{n-1}-\langle x_{i}(t),v\rangle)},
	\end{array}
	$$
	and	
	$$
	\begin{array}{l}
		\displaystyle{
			S_2=\frac{1}{N-1}\sum_{l:l\neq j}\alpha(t)\psi(x_{i}(t),x_{l}(t))(\langle x_{j}(t),v\rangle-m_{n-1})}\\
		\displaystyle{\hspace{0.7 cm}+\frac{1}{N-1}\sum_{l:l\neq j}\alpha(t)\psi(x_{i}(t),x_{l}(t))(m_{n-1}-\langle x_{l}(t),v\rangle)}\\
		\displaystyle{\hspace{0.4 cm}\leq K(\langle x_{j}(t),v\rangle-m_{n-1})+\frac{1}{N-1}\psi_{0}\alpha(t)\sum_{l:l\neq j}(m_{n-1}-\langle x_{l}(t),v\rangle).}
	\end{array}
	$$
	Hence, we have
	$$\begin{array}{l}
		\vspace{0.2cm}\displaystyle{\frac{d}{dt}\langle x_{i}(t)-x_{j}(t),v\rangle\leq K(M_{n-1}-m_{n-1}-\langle x_{i}(t)-x_{j}(t),v\rangle)}\\
		\vspace{0.2cm}\displaystyle{\hspace{1.8 cm}+\frac{1}{N-1}\psi_{0}\alpha(t)\sum_{l:l\neq i,j}(\langle x_{l}(t),v\rangle-M_{n-1}+m_{n-1}-\langle x_{l}(t),v\rangle)}\\
		\vspace{0.2cm}\displaystyle{\hspace{1.8 cm}+\frac{1}{N-1}\psi_{0}\alpha(t)(\langle x_{j}(t),v\rangle-M_{n-1}+m_{n-1}-\langle x_{i}(t),v\rangle)}\\
		\vspace{0.2cm}\displaystyle{\hspace{1.5 cm}=K(M_{n-1}-m_{n-1})-K\langle x_{i}(t)-x_{j}(t),v\rangle+\frac{N-2}{N-1}\psi_{0}\alpha(t)(-M_{n-1}+m_{n-1})}\\
		\vspace{0.2cm}\displaystyle{\hspace{1.8 cm}+\frac{1}{N-1}\psi_{0}\alpha(t)(\langle x_{j}(t),v\rangle-M_{n-1}+m_{n-1}-\langle x_{i}(t),v\rangle)}.
	\end{array}
	$$
	So, by taking into account of \eqref{posgennd} with $t\in [t_{n-1},t_{n}]$, we can write
	$$\begin{array}{l}
		\vspace{0.2cm}\displaystyle{
			\frac{d}{dt}\langle x_{i}(t)-x_{j}(t),v\rangle\leq K(M_{n-1}-m_{n-1})-K\langle x_{i}(t)-x_{j}(t),v\rangle
		}\\
		\vspace{0.4 cm}\displaystyle{\hspace{1.8 cm}
			+\frac{N-2}{N-1}\psi_{0}\alpha(t)(-M_{n-1}+m_{n-1})+\frac{1}{N-1}\psi_{0}\alpha(t)(-M_{n-1}+m_{n-1})
		}\\
		\vspace{0.4 cm}\displaystyle{\hspace{1.8 cm}
			-\frac{1}{N-1}\psi_{0}\alpha(t)\langle x_{i}(t)-x_{j}(t),v\rangle
		}\\
		\vspace{0.3 cm}\displaystyle{\hspace{1.5 cm}\leq K(M_{n-1}-m_{n-1})-K\langle x_{i}(t)-x_{j}(t),v\rangle+\psi_{0}\alpha(t)(-M_{n-1}+m_{n-1})
		}\\
		\displaystyle{\hspace{1.5 cm}=\left(K-\psi_{0}\alpha(t)\right)(M_{n-1}-m_{n-1})-K\langle x_{i}(t)-x_{j}(t),v\rangle.}
	\end{array}$$
	Hence, from Gronwall's inequality it turns out that $$\langle x_{i}(t)-x_{j}(t),v\rangle \leq e^{-K(t-t_{n-1})}\langle x_{i}(t_{n-1})-x_{j}(t_{n-1}),v\rangle$$$$+(M_{n-1}-m_{n-1})\int_{t_{n-1}}^{t}\left(K-\psi_{0}\alpha(s)\right)e^{-K(t-s)}ds,$$
	for all $t\in [t_{n-1},t_{n}]$. In particular, for $t=t_{n}$, from \eqref{distgen} it comes that 
	$$\begin{array}{l}
		\vspace{0.2cm}\displaystyle{d(t_{n})\leq e^{-K(t_{n}-t_{n-1})}\langle x_{i}(t_{n-1})-x_{j}(t_{n-1}),v\rangle+(M_{n-1}-m_{n-1})\int_{t_{n-1}}^{t_{n}}(K-\psi_{0}\alpha(s))e^{-K(t_{n}-s)}ds}\\
		\vspace{0.2cm}\displaystyle{\hspace{1.3cm}\leq e^{-K(t_{n}-t_{n-1})}\lvert x_{i}(t_{n-1})-x_{j}(t_{n-1})\rvert +(M_{n-1}-m_{n-1})\int_{t_{n-1}}^{t_{n}}(K-\psi_{0}\alpha(s))e^{-K(t_{n}-s)}ds}\\
		\vspace{0.2cm}\displaystyle{\hspace{1.3cm}\leq \left(e^{-K(t_{n}-t_{n-1})} +K\int_{t_{n-1}}^{t_{n}}e^{-K(t_{n}-s)}ds-\psi_{0}\int_{t_{n-1}}^{t_{n}}\alpha(s)e^{-K(t_{n}-s)}ds\right)d(t_{n-1})}\\
		\vspace{0.2cm}\displaystyle{\hspace{1.3cm}= \left(e^{-K(t_{n}-t_{n-1})} +1-e^{-K(t_{n}-t_{n-1})}-\psi_{0}\int_{t_{n-1}}^{t_{n}}\alpha(s)e^{-K(t_{n}-s)}ds\right)d(t_{n-1})}\\
		\vspace{0.2cm}\displaystyle{\hspace{1.3cm}=\left(1-\psi_{0}\int_{t_{n-1}}^{t_{n}}\alpha(s)e^{-K(t_{n}-s)}ds\right)d(t_{n-1}).}
	\end{array}$$
	Now, from \eqref{alpha1}, $$\psi_{0}\int_{t_{n-1}}^{t_{n}}\alpha(s)e^{-K(t_{n}-s)}ds\geq \psi_{0}e^{-KT}\int_{t_{n-1}}^{t_{n}}\alpha(s)ds\geq \psi_{0}e^{-KT}\bar{\alpha}.$$
	Then,
	$$d(t_{n})<(1-\psi_{0}e^{-KT}\bar{\alpha})d(t_{n-1})\leq (1-\psi_{0}e^{-KT}\bar{\alpha})d(t_{n-1}).$$
	So, if we set $$C:=\max\{1-e^{-KT},1-\psi_{0}e^{-KT}\tilde{\alpha}\},$$
	by taking into account of \eqref{t0gennd}, we can conclude that $C\in (0,1)$ is the constant for which \eqref{n-2gennd} holds.
\end{proof}
\begin{proof}[\textbf{Proof of Theorem \ref{consnd}}]
	Let $\{x_{i}\}_{i=1,\dots,N}$ be solution to \eqref{onoffnd}, \eqref{incondnd}. From 
\eqref{n-2gennd} we deduce
	\begin{equation}\label{11gennd}
		d(t_{n})\leq C^{n}d(0),\quad \forall n\geq 0.
	\end{equation}
	Note that  \eqref{11gennd} can be rewritten as  
	$$d(t_{n})\leq e^{-nTln\left(\frac{1}{C}\right)\frac{1}{T}}d(0).$$
Therefore, setting
	 $$\gamma=\frac{1}{T}\ln\left(\frac{1}{\tilde{C}}\right),$$
	we get
\begin{equation}\label{12gennd}
		d(t_{n})\leq e^{-nT\gamma }d(0),\quad \forall n\in\N_0.
	\end{equation}
	Now, fix $i,j=1,\dots,N$ and $t\geq 0.$ Then, $t\in [nT,(n+1)T]$, for some $n\in \mathbb{N}_0$. Therefore, by using \eqref{distgennd} with $t\geq nT= t_{0}+nT\geq t_{n}$ and \eqref{12gennd}, it turns out that 
	$$\lvert x_{i}(t)-x_{j}(t)\rvert\leq d(t_{n})\leq e^{-n\gamma T}d(0).$$
	Thus, since $t\leq (n+1)T$, we get
	$$\lvert x_{i}(t)-x_{j}(t)\rvert\leq e^{-\gamma t}\, e^{\gamma T}d(0)=e^{-\gamma(t-T)}d(0).$$
	Finally, by definition of the diameter $d(t)$, we can conclude that
	$$d(t)\leq e^{-\gamma(t-T)}d(0),\quad \forall t\ge 0,$$
	which proves the exponential decay estimate \eqref{estgennd}. 
\end{proof}
	\section{The Hegselmann-Krause model with time delay and attractive-lacking interaction}
Now, we examine the case in which the interactions between the agents of the system are described by the Hegselmann-Krause model
\begin{equation}\label{onoff}
	\frac{d}{dt}x_{i}(t)=\underset{j:j\neq i}{\sum}\alpha(t) b_{ij}(t)(x_{j}(t-\tau)-x_{i}(t)),\quad t>0,	\,\, \forall i=1,\dots,N,
\end{equation}
where the constant $\tau>0$ is the time delay. Although the case $\tau=0$ has been discussed in the previous section, we can include this situation in the analysis we will carry out throughout this section.
\\Here, the communication rates $b_{ij}$ are of the form
\begin{equation}\label{weight}
b_{ij}(t):=\frac{1}{N-1}\psi( x_{i}(t), x_{j}(t-\tau)), \quad\forall t>0,\, \forall i,j=1,\dots,N,
\end{equation}
where the influence function $\psi:\RR^d\times\RR^d\rightarrow \RR$ is positive, bounded and continuous and
\begin{equation}\label{K}
	K:=\lVert \psi\rVert_{\infty}.
\end{equation} 
We set\begin{equation}\label{M0}
	{\tilde M}^{0}:=\max_{i=1,\dots,N}\,\,\max_{s\in [-\tau, 0]}\lvert x_{i}(s)\rvert,
\end{equation} \begin{equation}\label{psi0}
	{\tilde\psi}_{0}:=\min_{\vert y\vert, \vert z\vert \le {\tilde M}^{0}}\psi(y,z).
\end{equation}
Let us note that, in the case $\tau=0$, the constants $\tilde{M}^{0}$, $\tilde{\psi}_{0}$ coincide with the positive constants $M^{0}$, $\psi_{0}$ given by \eqref{maxincond} and  \eqref{psi0nd}, respectively.
\\Again, the weight function $\alpha:[0,+\infty)\rightarrow[0,1]$ is continuous and satisfies property $(WF)$.

Note that, due to the presence of the time delay, now the initial conditions are no longer vectors in $\RR^d$ but are instead functions defined in the interval $[-\tau, 0].$

The initial conditions
\begin{equation}\label{incond}
x_{i}(s)=x^{0}_{i}(s),\quad \forall s\in [-\tau,0],\,\forall i=1,\dots,N,
\end{equation}
are assumed to be continuous functions. 

Now, we prove that each solution to system \eqref{onoff} converges exponentially to consensus. 
\begin{thm}\label{consgen}
	Let $\psi:\RR^d\times \RR^d\rightarrow\RR$ be a positive, bounded, continuous function. Let $x_{i}^{0}:[-\tau,0]\rightarrow\RR^{d}$ be a continuous function,  for any $i=1,\dots,N.$ Assume that the weight function $\alpha:[0,+\infty)\rightarrow[0,1]$ satisfies $(WF)$. Then, every solution $\{x_{i}\}_{i=1,\dots,N}$ to \eqref{onoff} with the initial conditions \eqref{incond} satisfies the exponential decay estimate 
	\begin{equation}\label{estgen}
		d(t)\leq \left(\max_{i,j=1,\dots,N}\,\,\max_{r,s\in [-\bar\tau,0]}\lvert x_{i}(r)-x_{j}(s)\rvert\right)e^{-\tilde{\gamma}(t-3T+\tau)},\quad \forall  t\geq 0,
	\end{equation}
where $T$ is the positive constant in \eqref{tnnd},
	for a suitable positive constant $\tilde{\gamma}$ independent of $N.$
\end{thm}
\subsection{Proof of Theorem \ref{consgen}}
Let $\{x_{i}\}_{i=1,\dots,N}$ be solution to \eqref{onoff} under the initial conditions \eqref{incond}. We assume that the hypotheses of Theorem \ref{consgen} are satisfied. In order to prove Theorem \eqref{consgen}, we first present some auxiliary lemmas.
\begin{lem}\label{L1}
	For each $v\in \RR^{d}$ and $S\ge 0,$  we have that 
	\begin{equation}\label{scalpr}
		\min_{j=1,\dots,N}\min_{s\in[S-\tau,S]}\langle x_{j}(s),v\rangle\leq \langle x_{i}(t),v\rangle\leq \max_{j=1,\dots,N}\max_{s\in[S-\tau,S]}\langle x_{j}(s),v\rangle,
	\end{equation}for all  $t\geq S-\tau$ and $i=1,\dots,N$. 
\end{lem}
\begin{proof}
	First of all, we note that the inequalities in \eqref{scalpr} are satisfied for every  $t\in [S-\tau,S]$.
	\\Now, let $S\geq 0$. Given a vector $v\in \RR^{d}$, we set $$M_S=\max_{j=1,\dots,N}\max_{s\in[S-\tau,S]}\langle x_{j}(s),v\rangle.$$
	For all $\epsilon >0$, we define
	$$K^{\epsilon}:=\left\{t>S :\max_{i=1,\dots,N}\langle x_{i}(s),v\rangle < M_S+\epsilon,\,\forall s\in [S,t)\right\}.$$
	By continuity, it holds $K^{\epsilon}\neq\emptyset$. Thus, denoted with $$S^{\epsilon}:=\sup K^{\epsilon},$$
	we have $S^{\epsilon}>T$. \\We claim that $S^{\epsilon}=+\infty$. Indeed, suppose by contradiction that $S^{\epsilon}<+\infty$. By definition of $S^{\epsilon}$, it turns out that \begin{equation}\label{max}
		\max_{i=1,\dots,N}\langle x_{i}(t),v\rangle<M_S+\epsilon,\quad \forall t\in (S,S^{\epsilon}),
	\end{equation}
	\begin{equation}\label{teps}
		\lim_{t\to S^{\epsilon-}}\max_{i=1,\dots,N}\langle x_{i}(t),v\rangle=M_S+\epsilon.
	\end{equation}
	For all $i=1,\dots,N$ and $t\in (S,S^{\epsilon})$, we have that
	$$\frac{d}{dt}\langle x_{i}(t),v\rangle=\frac{1}{N-1}\sum_{j:j\neq i}\alpha(t)\psi(x_{i}(t), x_{j}(t-\tau))\langle x_{j}(t-\tau)-x_{i}(t),v\rangle.$$
	Now, being $t\in (S,S^{\epsilon})$, $t-\tau\in (S-\tau, S^{\epsilon}-\tau)$ and from \eqref{max}
	\begin{equation}\label{t-tau}
		\langle x_{j}(t-\tau),v\rangle< M_S+\epsilon,\quad \forall j=1, \dots, N.
	\end{equation}
	Therefore, using \eqref{K}, \eqref{max}, \eqref{t-tau}, by recalling that $\alpha(t)\in [0,1]$, we can write $$\frac{d}{dt}\langle x_{i}(t),v\rangle\leq \frac{1}{N-1}\sum_{j:j\neq i}\alpha(t)\psi(x_{i}(t), x_{j}(t-\tau))(M_S+\epsilon-\langle x_{i}(t),v\rangle)$$$$\leq K(M_S+\epsilon-\langle x_{i}(t),v\rangle).$$
	Then, from Gronwall's inequality, we get
	$$\begin{array}{l}
		\vspace{0.2cm}\displaystyle{
			\langle x_{i}(t),v\rangle\leq e^{-K(t-S)}\langle x_{i}(S),v\rangle+K(M_S+\epsilon)\int_{S}^{t}e^{-K(t-s)}ds}\\
		\vspace{0.3cm}\displaystyle{\hspace{1.7 cm}
			=e^{-K(t-S)}\langle x_{i}(S),v\rangle+(M_S+\epsilon)e^{-Kt}(e^{Kt}-e^{KS})}\\
		\vspace{0.3cm}\displaystyle{\hspace{1.7 cm}
			=e^{-K(t-S)}\langle x_{i}(S),v\rangle+(M_S+\epsilon)(1-e^{-K(t-S)})}\\
		\vspace{0.3cm}\displaystyle{\hspace{1.7 cm}
			\leq e^{-K(t-S)}M_S+M_S+\epsilon -M_Se^{-K(t-S)}-\epsilon e^{-K(t-S)}}\\
		\vspace{0.3cm}\displaystyle{\hspace{1.7 cm}
			=M_S+\epsilon-\epsilon e^{-K(t-S)}}\\
		\displaystyle{\hspace{1.7 cm}
			\leq M_S+\epsilon-\epsilon e^{-K(S^{\epsilon}-S)},}
	\end{array}
	$$for all $t\in (S, S^{\epsilon})$.	We have so proved that, $\forall i=1,\dots, N,$
	$$\langle x_{i}(t),v\rangle\leq M_S+\epsilon-\epsilon e^{-K(S^{\epsilon}-S)}, \quad \forall t\in (S,S^{\epsilon}).$$
	Thus, we get
	\begin{equation}\label{lim}
		\max_{i=1,\dots,N} \langle x_{i}(t),v\rangle\leq M_S+\epsilon-\epsilon e^{-K(S^{\epsilon}-S)}, \quad \forall t\in (S,S^{\epsilon}).
	\end{equation}
	Letting $t\to S^{\epsilon-}$ in \eqref{lim}, from \eqref{teps} we have that $$M_S+\epsilon\leq M_S+\epsilon-\epsilon e^{-K(S^{\epsilon}-S)}<M_S+\epsilon,$$
	which is a contradiction. Thus, $S^{\epsilon}=+\infty$ and $$\max_{i=1,\dots,N}\langle x_{i}(t),v\rangle<M_S+\epsilon, \quad \forall t>S.$$
	From the arbitrariness of $\epsilon$ we can conclude that $$\max_{i=1,\dots,N}\langle x_{i}(t),v\rangle\leq M_S, \quad \forall t>S,$$
	from which $$\langle x_{i}(t),v\rangle\leq M_S, \quad \forall t>S, \,\forall i=1,\dots,N.$$
	So, the second inequality in \eqref{scalpr} is proved. Now, to show that the other inequality holds, given $v\in \RR^{d}$, we define $$m_S=\min_{j=1,\dots,N}\min_{s\in[S-\tau,S]}\langle x_{j}(s),v\rangle.$$
	Then, for all $i=1,\dots,N$ and $t>S$, by applying the second inequality in \eqref{scalpr} to the vector $-v\in\RR^{d}$ we get $$-\langle x_{j}(s),v\rangle=\langle x_{i}(t),-v\rangle\leq \max_{j=1,\dots,N}\max_{s\in[S-\tau,S]}\langle x_{j}(s),-v\rangle$$$$=-\min_{j=1,\dots,N}\min_{s\in [S-\tau,S]}\langle x_{j}(s),v\rangle=-m_S,$$
	from which $$\langle x_{j}(s),v\rangle\geq m_S.$$
	Thus, also the first inequality in \eqref{scalpr} is fulfilled.
\end{proof}
Now, we introduce some notation.
\begin{defn}We define, for all $n\in \mathbb{N}_{0}$,
	$$D_{n}:=\max_{i,j=1,\dots,N}\,\,\max_{s,t\in [t_{n}-\tau,t_{n} ]}\lvert x_{i}(s)-x_{j}(t)\rvert.$$
\end{defn}
In particular, being $t_{0}=0$,  $$D_{0}=\max_{i,j=1,\dots,N}\,\,\max_{s, t\in [-\tau, 0]}\lvert x_{i}(s)-x_{j}(t)\rvert.$$
Thus, the estimate \eqref{estgen} reduces to 
$$d(t)\leq e^{-\gamma(t-3T+\tau)}D_{0}.$$
The following lemma characterizes some properties of the sequence $\{D_{n}\}_{n\in\mathbb{N}_{0}}$.
\begin{lem}
	For each $n\in \mathbb{N}_{0}$ and $i,j=1,\dots,N$, we get \begin{equation}\label{distgen}
		\lvert x_{i}(s)-x_{j}(t)\rvert\leq D_{n}, \quad\forall s,t\geq t_{n}-\tau.
	\end{equation} 
\end{lem}
\begin{proof}
	Fix $n\in\mathbb{N}_{0}$ and $i,j=1,\dots,N$. Given $s,t\geq t_{n}-\tau$, if $\lvert x_{i}(s)-x_{j}(t)\rvert=0$, then of course $D_{n}\geq \lvert x_{i}(s)-x_{j}(t)\rvert$. Thus, we can assume $\lvert x_{i}(s)-x_{j}(t)\rvert>0$ and we set $$v=\frac{x_{i}(s)-x_{j}(t)}{\lvert x_{i}(s)-x_{j}(t)\rvert}.$$
	It turns out that $v$ is a unit vector and, by using \eqref{scalpr} with $S=t_{n}$ and the Cauchy-Schwarz inequality, we can write 
	$$\begin{array}{l}
		\vspace{0.3cm}\displaystyle{\lvert x_{i}(s)-x_{j}(t)\rvert=\langle x_{i}(s)-x_{j}(t),v\rangle=\langle x_{i}(s),v\rangle-\langle x_{j}(t),v\rangle}\\
		\vspace{0.3cm}\displaystyle{\leq \max_{l=1,\dots,N}\max_{r\in [t_{n}-\tau,t_{n}]}\langle x_{l}(r),v\rangle-\min_{l=1,\dots,N}\min_{r\in [t_{n}-\tau,t_{n}]}\langle x_{l}(r),v\rangle}\\
		\vspace{0.3cm}\displaystyle{\leq \max_{l,k=1,\dots,N}\max_{r,\sigma\in [t_{n}-\tau,t_{n}]}\langle x_{l}(r)-x_{k}(\sigma),v\rangle}	\\
		\displaystyle{\leq \max_{l,k=1,\dots,N}\max_{r,u\in [t_{n}-\tau,t_{n}]}\lvert x_{l}(r)-x_{k}(u)\rvert\lvert v\rvert=D_{n},}
	\end{array}$$
	which proves \eqref{distgen}.
\end{proof}
\begin{oss}
	Let us note that from \eqref{distgen} it follows that
	\begin{equation}
		\label{dx}
		\lvert x_{i}(s)-x_{j}(t)\rvert\leq D_{0}, \quad\forall s,t\geq -\tau.
	\end{equation}
	Moreover, it holds
	\begin{equation}\label{decgen}
		D_{n+1}\leq D_{n},\quad \forall n\in \mathbb{N}_{0}.
	\end{equation}
	Indeed, let $i,j=1,\dots,N$, $s,t\in [t_{n+1}-\tau,t_{n+1}]$ be such that $$D_{n+1}=\lvert x_{i}(s)-x_{j}(t)\rvert.$$
	Then, if $D_{n+1}=0$, of course we have $D_{n+1}\leq D_{n}$. On the other hand, if $D_{n+1}>0$, applying \eqref{distgen} with $s,t\geq t_{n+1}-\tau\geq t_{n}-\tau$, we get $$D_{n+1}=\lvert x_{i}(s)-x_{j}(t)\rvert\leq D_{n}.$$
\end{oss} 
With analogous arguments, one can fins a bound on $\vert x_i(t)\vert,$ uniform with respect to $t$ and $i=1,\dots,N.$ 
\begin{lem}\label{L3}
	For every $i=1,\dots,N,$ we have that \begin{equation}\label{boundsol}
		\lvert x_{i}(t)\rvert\leq \tilde{M}^{0},\quad \forall t\geq-\tau,
	\end{equation}
	where $\tilde{M}^{0}$ is that in \eqref{M0}.	
\end{lem}
\begin{proof}
	Given $i=1,\dots,N$ and $t\geq-\tau$, if $\lvert x_{i}(t)\rvert =0$, then trivially $\tilde{M}^{0}\geq \lvert x_{i}(t)\rvert $. On the contrary, if $\lvert x_{i}(t)\rvert >0$, we define $$v=\frac{x_{i}(t)}{\lvert x_{i}(t)\rvert},$$
	which is a unit vector for which it holds
	$$\lvert x_{i}(t)\rvert=\langle x_{i}(t),v\rangle. $$
	Then, by applying \eqref{scalpr} for $S=0$ and by using the Cauchy-Schwarz inequality, we get $$\lvert x_{i}(t)\rvert\leq \max_{j=1,\dots,N}\max_{s\in [-\tau,0]}\langle x_{j}(s),v\rangle\leq \max_{j=1,\dots,N}\max_{s\in [-\tau,0]}\lvert x_{j}(s)\rvert\lvert v\rvert$$$$=\max_{j=1,\dots,N}\max_{s\in [-\tau,0]}\lvert x_{j}(s)\rvert=\tilde{M}^{0},$$
	and \eqref{boundsol} is satisfied.
\end{proof}
\begin{oss}\label{R1}
	From the  estimate \eqref{boundsol}, since the influence function $\psi$ is continuous, we deduce that 
	\begin{equation}\label{stima_psi}
		\psi (x_i(t), x_j(t-\bar{\tau}))\ge \tilde{\psi}_{0},
	\end{equation}
	for all $t\ge 0,$ for all $i,j=1,\dots, N,$ where $\psi_{0}$ is the positive constant in \eqref{psi0}.
\end{oss}
\begin{lem}
	For all $i,j=1,\dots,N$,  unit vector $v\in \RR^{d}$ and $n\in\mathbb{N}_{0}$, we have
	\begin{equation}\label{4gen}
		\langle x_{i}(t)-x_{j}(t),v\rangle\leq e^{-K(t-\bar{t})}\langle x_{i}(\bar{t})-x_{j}(\bar{t}),v\rangle+(1-e^{-K(t-\bar{t})})D_{n},
	\end{equation}
	for all $t\geq \bar{t}\geq t_{n}$. \\Moreover, for all $n\in \mathbb{N}_{0}$, we get
	\begin{equation}\label{n+1gen}
		D_{n+1}\leq e^{-KT}d(t_{n})+(1-e^{-KT})D_{n}.
	\end{equation}
\end{lem}
\begin{proof}
		Fix $n\in\mathbb{N}_{0}$ and $v\in\RR^{d}$ such that $\lvert v\rvert=1$. We set $$M_n=\max_{i=1,\dots,N}\max_{t\in [t_{n}-\tau,t_{n}]}\langle x_{i}(t),v\rangle,$$$$m_n=\min_{i=1,\dots,N}\min_{t\in [t_{n}-\tau,t_{n}]}\langle x_{i}(t),v\rangle.$$
	Then,  $M_n-m_n\leq D_{n}$. In addition, for all $i=1,\dots,N$ and $t\geq \bar{t}\geq t_{n}$, from \eqref{scalpr} with $t-\tau\geq t_{n}-\tau$ we have that 
	$$
	\begin{array}{l}
		\displaystyle{
			\frac{d}{dt}\langle x_{i}(t),v\rangle=\sum_{j:j\neq i}\alpha(t)a_{ij}(t)\langle x_{j}(t-\tau)-x_{i}(t),v\rangle}\\
		\displaystyle{\hspace{2 cm}\leq \frac{1}{N-1}\sum_{j:j\neq i}\alpha(t)\psi(x_{i}(t), x_{j}(t-\tau))(M_n-\langle x_{i}(t),v\rangle).}
	\end{array}
	$$
	Therefore, since from \eqref{scalpr} $\langle x_{i}(t),v\rangle\leq M_n$ and since $\alpha(t)\in [0,1]$, it holds $$\frac{d}{dt}\langle x_{i}(t),v\rangle\leq K(M_n-\langle x_{i}(t),v\rangle).$$
	Thus, from Gronwall's inequality it follows that\begin{equation}\label{2gen}
		\langle x_{i}(t),v\rangle\leq e^{-K(t-\bar{t})}\langle x_{i}(\bar{t}),v\rangle+(1-e^{-K(t-\bar{t})})M_n.
	\end{equation}
	Employing analogous arguments, for all $i=1,\dots,N$ and $t\geq \bar{t}\geq t_{n}$, it turns out that
	$$\frac{d}{dt}\langle x_{i}(t),v\rangle\geq K(m_n-\langle x_{i}(t),v\rangle),$$
	and, applying again the Gronwall's inequality, we have that
	\begin{equation}\label{3gen}
		\langle x_{i}(t),v\rangle\geq e^{-K(t-\bar{t})}\langle x_{i}(\bar{t}),v\rangle+(1-e^{-K(t-\bar{t})})m_n.
	\end{equation}
	Therefore, for all $i,j=1,\dots,N$ and $t\geq \bar{t}\geq t_{n}$, from \eqref{2gen} and \eqref{3gen} we finally get
	$$
	\begin{array}{l}
		\vspace{0.3cm}\displaystyle{\langle x_{i}(t)-x_{j}(t),v\rangle\leq e^{-K(t-\bar{t})}\langle x_{i}(\bar{t})-x_{j}(t_{0}),v\rangle+(1-e^{-K(t-\bar{t})})D_{n},}
	\end{array}
	$$
	i.e. \eqref{4gen} holds true.
	\\Now we prove \eqref{n+1gen}. Given $n\in\mathbb{N}_{0}$, let $i,j=1,\dots,N$ and $s,t\in [t_{n+1}-\tau,t_{n+1}]$ be such that $D_{n+1}=\lvert x_{i}(s)-x_{j}(t)\rvert$. Note that, if $\lvert x_{i}(s)-x_{j}(t)\rvert=0$, then obviously 
	$$0=D_{n+1}\leq e^{-KT}d(t_{n})+(1-e^{-KT})D_{n}.$$ 
	So we can assume $\lvert x_{i}(s)-x_{j}(t)\rvert>0$. We define the unit vector $$v=\frac{x_{i}(s)-x_{j}(t)}{\lvert x_{i}(s)-x_{j}(t)\rvert}.$$
	Then, $$D_{n+1}=\langle x_{i}(s)-x_{j}(t),v\rangle=\langle x_{i}(s),v\rangle-\langle x_{j}(t),v\rangle.$$
	Now, from \eqref{tnnd} and \eqref{2gen} with $\bar{t}=t_{n}$, we have that
	\begin{equation}\label{5gen}
		\begin{split}
			\langle x_{i}(s),v\rangle&\leq e^{-K(s-t_{n})}(\langle x_{i}(t_{n}),v\rangle-M_n)+M_n\\&\leq e^{-KT}(\langle x_{i}(t_{n}),v\rangle-M_n)+M_n\hspace{1 cm}\\&\leq e^{-KT}\langle x_{i}(t_{n}),v\rangle+(1-e^{-KT})M_n.
		\end{split}
	\end{equation}	
	Similarly, it holds 
	\begin{equation}\label{6gen}
			\langle x_{j}(t),v\rangle\geq e^{-KT}\langle x_{j}(t_{n}),v\rangle+(1-e^{-KT})m_n.
	\end{equation}
	So, combining \eqref{5gen} and \eqref{6gen}, we can conclude that
	$$\begin{array}{l}
		\vspace{0.3cm}\displaystyle{D_{n+1}\leq e^{-KT}\langle x_{i}(t_{n})-x_{j}(t_{n}),v\rangle+(1-e^{-KT})(M_n-m_n)}\\
		\vspace{0.3cm}\displaystyle{\hspace{1cm}\leq e^{-KT}d(t_{n})+(1-e^{-KT})D_{n}.}
	\end{array}$$
\end{proof}
Finally, we need the following Lemma. Before stating it, we may assume, without loss of generality, that
\begin{equation}\label{tnbartau}
	\tau\leq t_{n}-t_{n-1}\leq T,
\end{equation}
where $T>0$ is the constant in \eqref{tnnd} or, eventually, some other constant bigger than the one in \eqref{tnnd}.
\\Let us note that if $\tau=0$ then \eqref{tnbartau} is obviously satisfied with $T$ being the positive constant given by \eqref{tnnd}.
\begin{lem}\label{L5gen}
	There exist a constant $C\in (0,1),$ independent of $N\in\N,$ such that
	\begin{equation}\label{n-2gen}
		d(t_{n})\leq C D_{n-2},\quad \forall n\geq 2.
	\end{equation} 
\end{lem}
\begin{proof}
	Given $n\geq 2$, since inequality \eqref{n-2gen} is trivially satisfied whenever $d(t_{n})=0$, we can suppose that $d(t_{n})>0$. Let $i,j=1,\dots,N$ be such that $d(t_{n})=\lvert x_{i}(t_{n})-x_{j}(t_{n})\rvert$. We set $$v=\frac{x_{i}(t_{n})-x_{j}(t_{n})}{\lvert x_{i}(t_{n})-x_{j}(t_{n})\rvert}.$$
	Then, $v$ is a unit vector for which we can write $$d(t_{n})=\langle x_{i}(t_{n})-x_{j}(t_{n}),v\rangle.$$
	We define $$M_{n-1}=\max_{l=1,\dots,N}\max_{s\in [t_{n-1}-\tau,t_{n-1}]}\langle x_{l}(s),v\rangle,$$
	$$m_{n-1}=\min_{l=1,\dots,N}\min_{s\in [t_{n-1}-\tau,t_{n-1}]}\langle x_{l}(s),v\rangle.$$
	Then $M_{n-1}-m_{n-1}\leq D_{n-1}$. 
	\\Now, we distinguish two different situations.
	\par\textit{Case I.} Assume that there exists $\bar{t}\in [t_{n-1}-\tau,t_{n}]$ such that 
	$$\langle x_{i}(\bar{t})-x_{j}(\bar{t}),v\rangle<0.$$
	Then, since from \eqref{tnbartau} we have that $t_{n}\geq\bar{t}\geq t_{n-1}-\tau\geq t_{n-2}$, we can apply \eqref{4gen} and we get \begin{equation}\label{t0gen}
		\begin{split}
			d(t_{n})&\leq e^{-K(t_{n}-\bar{t})}\langle x_{i}(\bar{t})-x_{j}(\bar{t}),v\rangle+(1-e^{-K(t_{n}-\bar{t})})D_{n-2}\\&\leq (1-e^{-K(t_{n}-\bar{t})})D_{n-2}\\&\leq (1-e^{-K(T+\bar{\tau})})D_{n-2}.
		\end{split}
	\end{equation}
	\par\textit{Case II.} Assume it rather holds \begin{equation}\label{posgen}
		\langle x_{i}(t)-x_{j}(t),v\rangle\geq 0,\quad \forall t\in [t_{n-1}-\tau,t_{n}].
	\end{equation}
	Then, for every  $t\in [t_{n-1},t_{n}]$, we have that 
	$$
	\begin{array}{l}
		\displaystyle{
			\frac{d}{dt}\langle x_{i}(t)-x_{j}(t),v\rangle=\frac{1}{N-1}\sum_{l:l\neq i}\alpha(t)\psi(x_{i}(t),x_{l}(t-\tau))\langle x_{l}(t-\tau)-x_{i}(t),v\rangle}\\
		\displaystyle{\hspace{1.1cm}-\frac{1}{N-1}\sum_{l:l\neq j}\alpha(t)\psi(x_{i}(t),x_{l}(t-\tau))\langle x_{l}(t-\tau)-x_{j}(t),v\rangle}\\
		\displaystyle{\hspace{0.6cm}=\frac{1}{N-1}\sum_{l:l\neq i}\alpha(t)\psi(x_{i}(t),x_{l}(t-\tau))(\langle x_{l}(t-\tau),v\rangle-M_{n-1}+M_{n-1}-\langle x_{i}(t),v\rangle)}\\
		\displaystyle{\hspace{1.1cm}+\frac{1}{N-1}\sum_{l:l\neq j}\alpha(t)\psi(x_{i}(t),x_{l}(t-\tau))(\langle x_{j}(t),v\rangle-m_{n-1}+m_{n-1}-\langle x_{l}(t-\tau),v\rangle)}\\
		\displaystyle{\hspace{5.5 cm}:=S_1+S_2.}
	\end{array}
	$$
	Now, being $t\in [t_{n-1},t_{n}]$, it holds that both $t,t-\tau\geq t_{n-1}-\tau$ and from \eqref{scalpr} we have that
	\begin{equation}\label{7gen}
		m_{n-1}\leq\langle x_{k}(t),v\rangle\leq M_{n-1}, \quad m_{n-1}\leq\langle x_{k}(t-\tau),v\rangle\leq M_{n-1},\quad \forall k=1,\dots,N.
	\end{equation}
	Therefore, using \eqref{boundsol}, we get
	$$
	\begin{array}{l}
		\displaystyle{
			S_1=\frac{1}{N-1}\sum_{l:l\neq i}\alpha(t)\psi(x_{i}(t),x_{l}(t-\tau))(\langle x_{l}(t-\tau),v\rangle -M_{n-1})}\\
		\displaystyle{\hspace{0.7 cm}+\frac{1}{N-1}\sum_{l:l\neq i}\alpha(t)\psi(x_{i}(t),x_{l}(t-\tau))(M_{n-1}-\langle x_{i}(t),v\rangle)}\\
		\displaystyle{\hspace{0.4 cm}\leq \frac{1}{N-1}\tilde{\psi}_{0}\alpha(t)\sum_{l:l\neq i}(\langle x_{l}(t-\tau),v\rangle-M_{n-1)}+K(M_{n-1}-\langle x_{i}(t),v\rangle)},
	\end{array}
	$$
	and	
	$$
	\begin{array}{l}
		\displaystyle{
			S_2=\frac{1}{N-1}\sum_{l:l\neq j}\alpha(t)\psi(x_{i}(t),x_{l}(t-\tau))(\langle x_{j}(t),v\rangle-m_{n-1})}\\
		\displaystyle{\hspace{0.7 cm}+\frac{1}{N-1}\sum_{l:l\neq j}\alpha(t)\psi(x_{i}(t),x_{l}(t-\tau))(m_{n-1}-\langle x_{l}(t-\tau),v\rangle)}\\
		\displaystyle{\hspace{0.4 cm}\leq K(\langle x_{j}(t),v\rangle-m_{n-1})+\frac{1}{N-1}\tilde{\psi}_{0}\alpha(t)\sum_{l:l\neq j}(m_{n-1}-\langle x_{l}(t-\tau),v\rangle).}
	\end{array}
	$$
	Combining this last fact with \eqref{7gen} it comes that 
	$$\begin{array}{l}
		\vspace{0.2cm}\displaystyle{\frac{d}{dt}\langle x_{i}(t)-x_{j}(t),v\rangle\leq K(M_{n-1}-m_{n-1}-\langle x_{i}(t)-x_{j}(t),v\rangle)}\\
		\vspace{0.2cm}\displaystyle{\hspace{1.8 cm}+\frac{1}{N-1}\tilde{\psi}_{0}\alpha(t)\sum_{l:l\neq i,j}(\langle x_{l}(t-\tau),v\rangle-M_{n-1}+m_{n-1}-\langle x_{l}(t-\tau),v\rangle)}\\
		\vspace{0.2cm}\displaystyle{\hspace{1.8 cm}+\frac{1}{N-1}\tilde{\psi}_{0}\alpha(t)(\langle x_{j}(t-\tau),v\rangle-M_{n-1}+m_{n-1}-\langle x_{i}(t-\tau),v\rangle)}\\
		\vspace{0.2cm}\displaystyle{\hspace{1.5 cm}=K(M_{n-1}-m_{n-1})-K\langle x_{i}(t)-x_{j}(t),v\rangle+\frac{N-2}{N-1}\tilde{\psi}_{0}\alpha(t)(-M_{n-1}+m_{n-1})}\\
		\vspace{0.2cm}\displaystyle{\hspace{1.8 cm}+\frac{1}{N-1}\tilde{\psi}_{0}\alpha(t)(\langle x_{j}(t-\tau),v\rangle-M_{n-1}+m_{n-1}-\langle x_{i}(t-\tau),v\rangle)}.
	\end{array}
	$$
	Now, since $t-\tau\in [t_{n-1}-\tau,t_{n}]$, from \eqref{posgen} it comes that $\langle x_{i}(t-\tau)-x_{j}(t-\tau),v\rangle\geq 0$. Then, we get
	$$\begin{array}{l}
		\vspace{0.2cm}\displaystyle{
			\frac{d}{dt}\langle x_{i}(t)-x_{j}(t),v\rangle\leq K(M_{n-1}-m_{n-1})-K\langle x_{i}(t)-x_{j}(t),v\rangle
		}\\
		\vspace{0.4 cm}\displaystyle{\hspace{1.8 cm}
			+\frac{N-2}{N-1}\tilde{\psi}_{0}\alpha(t)(-M_{n-1}+m_{n-1})+\frac{1}{N-1}\tilde{\psi}_{0}\alpha(t)(-M_{n-1}+m_{n-1})
		}\\
		\vspace{0.4 cm}\displaystyle{\hspace{1.8 cm}
			-\frac{1}{N-1}\tilde{\psi}_{0}\alpha(t)\langle x_{i}(t-\tau)-x_{j}(t-\tau),v\rangle
		}\\
		\vspace{0.3 cm}\displaystyle{\hspace{1.5 cm}\leq K(M_{n-1}-m_{n-1})-K\langle x_{i}(t)-x_{j}(t),v\rangle+\tilde{\psi}_{0}\alpha(t)(-M_{n-1}+m_{n-1})
		}\\
		\displaystyle{\hspace{1.5 cm}=\left(K-\tilde{\psi}_{0}\alpha(t)\right)(M_{n-1}-m_{n-1})-K\langle x_{i}(t)-x_{j}(t),v\rangle.}
	\end{array}$$
	Hence, from Gronwall's inequality it turns out that $$\langle x_{i}(t)-x_{j}(t),v\rangle \leq e^{-K(t-t_{n-1})}\langle x_{i}(t_{n-1})-x_{j}(t_{n-1}),v\rangle$$$$+(M_{n-1}-m_{n-1})\int_{t_{n-1}}^{t}\left(K-\tilde{\psi}_{0}\alpha(s)\right)e^{-K(t-s)}ds,$$
	for all $t\in [t_{n-1},t_{n}]$. In particular, for $t=t_{n}$, from \eqref{distgen} it comes that 
	$$\begin{array}{l}
		\vspace{0.2cm}\displaystyle{d(t_{n})\leq e^{-K(t_{n}-t_{n-1})}\langle x_{i}(t_{n-1})-x_{j}(t_{n-1}),v\rangle+(M_{n-1}-m_{n-1})\int_{t_{n-1}}^{t_{n}}(K-\tilde{\psi}_{0}\alpha(s))e^{-K(t_{n}-s)}ds}\\
		\vspace{0.2cm}\displaystyle{\hspace{1.3cm}\leq e^{-K(t_{n}-t_{n-1})}\lvert x_{i}(t_{n-1})-x_{j}(t_{n-1})\rvert +(M_{n-1}-m_{n-1})\int_{t_{n-1}}^{t_{n}}(K-\tilde{\psi}_{0}\alpha(s))e^{-K(t_{n}-s)}ds}\\
		\vspace{0.2cm}\displaystyle{\hspace{1.3cm}\leq \left(e^{-K(t_{n}-t_{n-1})} +K\int_{t_{n-1}}^{t_{n}}e^{-K(t_{n}-s)}ds-\tilde{\psi}_{0}\int_{t_{n-1}}^{t_{n}}\alpha(s)e^{-K(t_{n}-s)}ds\right)D_{n-1}}\\
		\vspace{0.2cm}\displaystyle{\hspace{1.3cm}= \left(e^{-K(t_{n}-t_{n-1})} +1-e^{-K(t_{n}-t_{n-1})}-\tilde{\psi}_{0}\int_{t_{n-1}}^{t_{n}}\alpha(s)e^{-K(t_{n}-s)}ds\right)D_{n-1}}\\
		\vspace{0.2cm}\displaystyle{\hspace{1.3cm}=\left(1-\tilde{\psi}_{0}\int_{t_{n-1}}^{t_{n}}\alpha(s)e^{-K(t_{n}-s)}ds\right)D_{n-1}.}
	\end{array}$$
	Now, from \eqref{alpha1}, $$\tilde{\psi}_{0}\int_{t_{n-1}}^{t_{n}}\alpha(s)e^{-K(t_{n}-s)}ds\geq \tilde{\psi}_{0}e^{-KT}\int_{t_{n-1}}^{t_{n}}\alpha(s)ds\geq \tilde{\psi}_{0}e^{-KT}\bar{\alpha}.$$
	Then, combining this last fact with \eqref{decgen},
	$$d(t_{n})<(1-\tilde{\psi}_{0}e^{-KT}\bar{\alpha})D_{n-1}\leq (1-\tilde{\psi}_{0}e^{-KT}\bar{\alpha})D_{n-1}.$$
	So, if we set $$C:=\max\{1-e^{-K(T+\tau)},1-\tilde{\psi}_{0}e^{-KT}\tilde{\alpha}\},$$
	by taking into account of \eqref{t0gen}, we can conclude $C\in (0,1)$ is the constant for which \eqref{n-2gen} holds.
\end{proof}
\begin{oss}
	Let us note that, if $\tau=0$, the the sequence $\{D_{n}\}_{n\in\mathbb{N}_{0}}$ coincides with the sequence of diameters $\{d(t_{n})\}_{n\in\mathbb{N}_{0}}$. Then, from Lemma \ref{L5gen}, $$d(t_{n})\leq Cd(t_{n-2}),\quad \forall n\ge 2,$$
	where $C=\max\{1-e^{-KT},1-\psi_{0}e^{-KT}\bar{\alpha}\}$ with $\psi_{0}$ being the positive constant in \eqref{psi0nd}. 
	\\However, in the undelayed case we have seen that, with a simpler proof than the one just presented, it holds $$d(t_{n})\leq Cd(t_{n-1}),\quad \forall n\ge 1.$$
\end{oss}
We are now able to prove Theorem \ref{consgen}.
\begin{proof}[\textbf{Proof of Theorem \ref{consgen}}]
		Let $(x_{i})_{i=1,\dots,N}$ be solution to \eqref{onoff}, \eqref{incond}. We claim that 
	\begin{equation}\label{10gen}
		D_{n+1}\leq \tilde{C}D_{n-2},\quad \forall n\geq2,
	\end{equation} 
	for some constant $\tilde{C}\in (0,1)$. Indeed, given $n\geq2$, from \eqref{decgen}, \eqref{n+1gen} and \eqref{n-2gen} we have that $$\begin{array}{l}
		\vspace{0.3cm}\displaystyle{D_{n+1}\leq e^{-KT}d(t_{n})+(1-e^{-KT})D_{n}}\\
		\vspace{0.3cm}\displaystyle{\hspace{1cm}\leq e^{-KT}CD_{n-2}+(1-e^{-KT})D_{n}}\\
		\vspace{0.3cm}\displaystyle{\hspace{1cm}\leq e^{-KT}CD_{n-2}+(1-e^{-KT})D_{n-2}}\\
		\displaystyle{\hspace{1cm}\leq(1-e^{-KT}(1-C)) D_{n-2},}
	\end{array}$$
	where the constant $C\in (0,1)$ is the constant in \eqref{n-2gen}.
	So, setting $$\tilde{C}=1-e^{-KT}(1-C),$$
	we can conclude that $\tilde{C}\in (0,1)$ is the constant for which \eqref{10gen} holds true.
	\\This implies that \begin{equation}\label{11gen}
		D_{3n}\leq \tilde{C}^{n}D_{0},\quad \forall n\geq 0.
	\end{equation}
	Indeed, by induction, if $n=0$ the claim is trivially satisfied.
	So, assume that \eqref{11gen} holds for some $n\geq 0$ and we prove it for $n+1$. By using again \eqref{10gen} and from the induction hypothesis it comes that $$D_{3(n+1)}=D_{3n+3}=D_{(3n+2)+1}\leq\tilde{C}D_{3n+2-2}=\tilde{C}D_{3n}\leq \tilde{C}\tilde{C}^{n}D_{0}=\tilde{C}^{n+1}D_{0},$$
	i.e. \eqref{11gen} is fulfilled.
	\\Notice that, from \eqref{11gen}, we have that 
	$$\displaystyle{D_{3n}\leq\left  (\sqrt[3]{\tilde{C}}\right )^{3n}D_{0}=e^{(\bar{n}+1)n\ln\left(\sqrt[(\bar{n}+1)]{\tilde{C}}\right )}D_{0}=e^{3nT\ln\left(\tilde{C}\right)\frac{1}{3T}}D_{0}, \quad \forall n\in\N.}$$
	Therefore, if we set $$\tilde{\gamma}=\frac{1}{3T}\ln\left(\frac{1}{\tilde{C}}\right),$$
	we can write \begin{equation}\label{12gen}
		D_{3n}\leq e^{-3n\tilde{\gamma} T}D_{0},\quad \forall n\in\N_0.
	\end{equation}
	Now, fix $i,j=1,\dots,N$ and $t\geq-\bar{\tau}.$ Then, $t\in [3nT-\tau, (3n+3)T-\tau]$, for some $n\in \mathbb{N}_0$. Therefore, by using \eqref{distgen} with $t\geq 3nT-\tau\geq t_{3n}-\tau$ and \eqref{12gen}, it turns out that 
	$$\lvert x_{i}(t)-x_{j}(t)\rvert\leq D_{3n}\leq e^{-3n\tilde{\gamma} T}D_{0}= e^{-\tilde{\gamma}(3nT-t)}e^{-\tilde{\gamma} t}D_{0}.$$
	Thus, being $t\leq (3n+3)T-\tau$, we get
	$$\lvert x_{i}(t)-x_{j}(t)\rvert\leq e^{-\tilde{\gamma} t}\, e^{\tilde{\gamma}(3T-\tau)}D_{0}.$$
	Then,
	$$d(t)\leq e^{-\tilde{\gamma} (t-3T+\tau)}D_0,\quad \forall t\ge 0$$
	and \eqref{estgen} is proved. 
\end{proof}
\begin{oss}
	Let us note that if $\tau=0$, then the constant $\tilde{C}\in (0,1)$ for which \eqref{10gen} holds can be chosen equal to the constant $C$ in \eqref{n-2gennd}. Thus, \eqref{estgen} yields
	$$d(t)\leq d(0)e^{-\tilde{\gamma}(t-3T)},\quad\forall t\geq 0,$$
	where $T>0$ is the positive number in \eqref{tnnd} and $$\tilde{\gamma}=\frac{1}{3T}\ln\left(\frac{1}{C}\right).$$
	Then, $$d(t)\leq d(0)e^{-\ln\left(\frac{1}{C}\right)\left(\frac{t}{3T}-1\right)}.$$
	On the other hand, Theorem \ref{consnd} guarantees that $$d(t)\leq d(0)e^{-\ln\left(\frac{1}{C}\right)\left(\frac{t}{T}-1\right)},\quad\forall t\geq 0.$$
	Therefore, being $e^{-\ln\left(\frac{1}{C}\right)\left(\frac{t}{T}-1\right)}\leq e^{-\ln\left(\frac{1}{C}\right)\left(\frac{t}{3T}-1\right)}$, Theorem \ref{consnd} is not redundant since \eqref{estgennd} improves \eqref{estgen} in the case $\tau=0$.
	
\end{oss}
\section{The continuum model}
In this section, we consider the continuum model obtained as the mean-field limit of the particle system
when $N\rightarrow \infty.$ Let $\mathcal{M}(\RR^{d})$ be the set of probability measures on the space $\RR^{d}$. Then, the continuum model associated with the particle system \eqref{onoff} is given by
\begin{equation}\label{kinetic}
	\begin{array}{l}
		\displaystyle{\partial_t \mu_t+ \mbox{\rm div\,} (F[\mu_{t-\tau}]\mu_t)=0, \quad t>0,} 
		\\
		\displaystyle{\mu_s=g_s, \quad x\in \RR^d, \quad  s\in[-\tau,0],}
	\end{array}
\end{equation}
where the velocity field $F$ is defined as 
\begin{equation}
	\label{F}
	F[\mu_{t-\tau}](x)=\int_{\RR^d} \alpha(t)\psi (x, y)(y-x)\,d\mu_{t-\tau}(y),
\end{equation}
and $g_s \in \mathcal{C}([-\tau,0];\mathcal{M}(\RR^d))$.  

We assume that the potential $\psi(\cdot,\cdot)$ in \eqref{F} is Lipschitz continuous, namely there exists $L>0$ such that, for any $(x,y),(x',y')\in\RR^{2d}$, we have 
$$
|\psi(x,y)-\psi(x', y')|\leq L(|y-y'|+|x-x'|).
$$
\begin{defn}
	Let $T>0$. We say that $\mu_t\in \mathcal{C}([0,T);\mathcal{M}(\RR^d))$ is a measure-valued solution to \eqref{kinetic} on the time interval $[0,T)$ if for all $\varphi\in \mathcal{C}^\infty_c(\RR^d\times [0,T))$ we have:
	\begin{equation}
		\label{weak_solution}
		\int_0^T \int_{\RR^d} \left( \partial_t\varphi+F[\mu_{t-\tau(t)}](x) \cdot \nabla_x\varphi \right) d\mu_t(x) dt +\int_{\RR^d} \varphi (x,0)dg_0(x) =0.
	\end{equation}
\end{defn}
Before stating the consensus result for solutions to model \eqref{kinetic}, we recall some
basic tools on probability spaces and measures.
\begin{defn}
	Let $\mu,\nu\in \mathcal{M}(\RR^d)$ be two probability measures on $\RR^d$. We define the 1-Wasserstein distance between $\mu$ and $\nu$ as
	$$
	d_1(\mu,\nu):=\inf_{\pi\in \Pi(\mu,\nu)} \int_{\RR^d\times \RR^d} |x-y|d\pi(x,y),
	$$
	where $\Pi (\mu,\nu)$ is the space of all couplings for $\mu$ and $\nu$, namely all those probability measures on $\RR^{2d}$ having as marginals $\mu$ and $\nu$:
	$$
	\int_{\RR^d\times\RR^d} \varphi (x)d\pi(x,y)=\int_{\RR^d} \varphi(x)d\mu(x), \quad \int_{\RR^d\times\RR^d} \varphi(y)d\pi(x,y)=\int_{\RR^d}\varphi(y)d\nu(y),
	$$
	for all $\varphi \in \mathcal{C}_b(\RR^d)$.
\end{defn}
Let us introduce the space $\mathcal{P}_1$ of all probability measures with finite first-order moment. It is  well-known that $(\mathcal{P}_1(\RR^d),d_1(\cdot,\cdot))$ is a complete metric space.

Now, we define the position diameter for a compactly supported measure $g \in {\mathcal{P}}_1(\RR^d)$ as follows:
$$d_X[g]:=\mbox{diam}(supp\ g).$$

Since the consensus result for the particle model \eqref{onoff} holds without any upper bounds on the time delay $\tau$, one can prove the following consensus theorem for the PDE model \eqref{kinetic} without requiring a smallness assumption on the time delay $\tau.$
We omit the proof since, once we have the result for the particle system \eqref{onoff} with estimates independent of the number of agents, the consensus estimate for the continuum model is obtained with arguments analogous to the ones in \cite{CPP} and \cite{PPpreprint} in the more general case of a time-dependent time delay.

\begin{thm}
	Let $\mu_t\in \mathcal{C}([0,T];\mathcal{P}_1(\RR^d))$ be a measure-valued solution to \eqref{kinetic} with compactly supported initial datum $g_s\in \mathcal{C}([-\tau,0];\mathcal{P}_1(\RR^d))$ and let $F$ as in \eqref{F}.
	Then, there exists a constant $C>0$  such that 
	$$
	d_X(\mu_t)\leq \left( \max_{s\in[-\tau,0]} d_X(g_s) \right) e^{-Ct},
	$$
	for all $t\ge 0.$
\end{thm}

\end{document}